\documentclass[a4paper]{amsart}

\usepackage[T1]{fontenc}
\usepackage{amsmath, amssymb, amsthm, mathrsfs, mathtools, marvosym, graphicx}
\usepackage{varioref}
\usepackage{hyperref}
\usepackage{enumitem, xspace, ifthen, comment}
\usepackage[all]{xy}
\usepackage[dvipsnames]{xcolor}
\usepackage[nameinlink, capitalize]{cleveref}
\usepackage{tikz}
\usetikzlibrary{shapes, calc, decorations.pathreplacing}

\setlist[enumerate]{
  label=(\thethm.\arabic*),
  before={\setcounter{enumi}{\value{equation}}},
  after={\setcounter{equation}{\value{enumi}}},
  itemsep=1ex
}

\setlist[itemize]{
  leftmargin=*,
  itemsep=1ex,
  label=$\circ$
}

\newtheorem*{thm-plain}{Theorem}
\newtheorem{thm}{Theorem}
\newtheorem{lem}[thm]{Lemma}
\newtheorem{prp}[thm]{Proposition}

\theoremstyle{definition}

\newtheorem*{dfn-plain}{Definition}

\theoremstyle{remark}

\newtheorem*{rem-plain}{Remark}

\newcommand{\inv}{^{-1}}
\newcommand{\from}{\colon}

\newcommand{\lto}{\longrightarrow}

\newcommand{\inj}{\hookrightarrow}

\newcommand{\isom}{\cong}
\newcommand{\defn}{\coloneqq}

\newcommand{\tensor}{\otimes}

\newcommand{\wt}{\widetilde}

\newcommand{\factor}[2]{\left. \raise 2pt\hbox{$#1$} \right/\hskip -2pt \raise -2pt\hbox{$#2$}}

\newdir{ ir}{{}*!/-5pt/@^{(}} 
\newdir{ il}{{}*!/-5pt/@_{(}} 

\DeclareMathOperator{\rk}{rk}

\newcommand{\set}[1]{\left\{ #1 \right\}}

\def\rd#1.{\lfloor{#1}\rfloor}
\def\rp#1.{\lceil{#1}\rceil}
\def\tw#1.{\langle{#1}\rangle}

\renewcommand{\O}[1]{\mathscr{O}_{#1}}
\newcommand{\Omegap}[2]{\Omega_{#1}^{#2}}

\newcommand{\Omegal}[3]{\Omega_{#1}^{#2} \big( \!\log #3 \big)}

\newcommand{\Reg}[1]{{#1}_{\mathrm{reg}}}
\newcommand{\Sing}[1]{{#1}_{\mathrm{sg}}}

\newcommand{\codim}[2]{\mathrm{codim}_{#1}(#2)}

\def\Hnought#1.#2.{\mathit{\Gamma} \!\left( #1, #2 \right)}
\def\HH#1.#2.#3.{\mathrm{H}^{#1} \!\left( #2, #3 \right)}
\def\euler#1.#2.{\chi \!\left( #1, #2 \right)}
\def\HHbig#1.#2.#3.{\mathrm{H}^{#1} \big( #2, #3 \big)}
\def\hh#1.#2.#3.{h^{#1} \!\left( #2, #3 \right)}
\def\RR#1.#2.#3.{R^{#1} #2_* #3}
\def\HHc#1.#2.#3.{\mathrm{H}_{\mathrm{c}}^{#1} \!\left( #2, #3 \right)}
\def\Hh#1.#2.#3.{\mathrm{H}_{#1} \!\left( #2, #3 \right)}
\def\Hom#1.#2.#3.{\mathrm{Hom}_{#1} \!\left( #2, #3 \right)}
\def\sHom#1.#2.{\mathscr{H}\!om \!\left( #1, #2 \right)}
\def\Ext#1.#2.#3.{\mathrm{Ext}^{#1} \!\left( #2, #3 \right)}
\def\sExt#1.#2.#3.{\mathscr{E}\!xt^{#1} \!\left( #2, #3 \right)}

\DeclareMathOperator{\Spec}{Spec}

\DeclareMathOperator{\res}{res}

\DeclareMathOperator{\supp}{supp}

\newcommand{\germ}[2]{\left( #1 \in #2 \right)} 
\newcommand{\zp}{\ensuremath{\mathbb Z_{(p)}}}

\renewcommand{\theta}{\vartheta}
\renewcommand{\phi}{\varphi}

\newcommand{\Z}{\ensuremath{\mathbb Z}}
\newcommand{\Q}{\ensuremath{\mathbb Q}}

\newcommand{\sG}{\mathscr G}  
  \newcommand{\sL}{\mathscr L}

\definecolor{forrest}{RGB}{81,133,49}
\definecolor{mydarkblue}{RGB}{10,92,153}

\title{A note on Flenner's Extension Theorem}

\dedicatory{\raggedright
\hspace{.332\linewidth} Now I know you are not there. \\
\hspace{.332\linewidth} We will not see you again. \\
\hspace{.332\linewidth} At least with our eyes. \\
\hspace{.332\linewidth} You are out of the dark, \\
\hspace{.332\linewidth} out of space.}

\author{Patrick Graf}
\address{Lehrstuhl f\"ur Mathematik I, Universit\"at Bayreuth, 95440 Bayreuth, Germany}
\email{\href{mailto:patrick.graf@uni-bayreuth.de}{patrick.graf@uni-bayreuth.de}}
\urladdr{\href{http://www.pgraf.uni-bayreuth.de/en/}{www.graficland.uni-bayreuth.de}}

\date{May 6, 2019}
\thanks{The author was supported in full by a DFG Research Fellowship. This research was carried out at the University of Utah.}

\hypersetup{
  pdfauthor={Patrick Graf},
  pdftitle={A note on Flenner's Extension Theorem},
  pdfkeywords={Reflexive differentials, differential forms on the smooth locus, extension theorem, high-codimension singularities, singularities in positive characteristic},
  pdfstartview={Fit},
  pdfpagelayout={OneColumn},
  pdfpagemode={UseNone},
  linktoc=all,
  breaklinks,
  linkcolor=[RGB]{0 0 96},
  citecolor=[RGB]{96 0 0},
  urlcolor=[RGB]{0 96 0},
  colorlinks}

\begin{document}

\begin{abstract}
We show that any $p$-form on the smooth locus of a normal complex space extends to a resolution of singularities, possibly with logarithmic poles, as long as $p \le \codim X{\Sing X} - 2$.
A stronger version of this result, allowing no poles at all, is originally due to Flenner.
Our proof, however, is not only completely different, but also shorter and technically simpler.
We furthermore give examples to show that the statement fails in positive characteristic.
\end{abstract}

\maketitle

\thispagestyle{empty}

Let $X$ be a normal complex space and $\pi \from Y \to X$ a resolution of singularities.
It is an important question whether reflexive differentials, i.e.~differential forms on the smooth locus $\Reg X$ of $X$, extend holomorphically to forms on $Y$.
Two important papers proving such extension properties under additional assumptions on the singularities of $X$ are~\cite{GKKP11} and the recent~\cite{KebekusSchnell18}.
In a slightly different direction, Ohsawa~\cite{Ohsawa82} and independently Steenbrink--van Straten~\cite{SvS85} showed that extension always holds if $X$ has isolated singularities and $p \le \dim X - 2$.
Later, Flenner~\cite{Fle88} generalized this to not necessarily isolated singularities as follows:

\begin{thm}[Flenner's Extension Theorem] \label{flenner}
Let $X$ be a germ of a normal complex space (or even more generally, the spectrum of a normal complete algebra over a field of characteristic zero) and $\pi \from Y \to X$ a resolution.
Then the canonical map
\[ \HHbig0.Y.\Omegap Yp. \lto \HHbig0.\Reg X.\Omegap Xp. \]
is bijective for all $p \le \codim X{\Sing X} - 2$.
\end{thm}

Both~\cite{Ohsawa82, SvS85} and~\cite{Fle88} heavily rely in their proofs on results about the Hodge theory of isolated singularities.
The purpose of this short note is twofold: Firstly, to give a novel and technically significantly simpler proof of the following logarithmic version of \cref{flenner}.
While this variant is obviously somewhat weaker, we point out that by combining it with~\cite[Thm.~3.1]{GK13}, at least for $1$-forms we do obtain a new (and still simpler) proof of \cref{flenner} in full strength.

\begin{thm}[Logarithmic Flenner theorem] \label{flenner log}
Let $X$ be a germ of a normal complex space and $\pi \from Y \to X$ a log resolution, with exceptional locus $E \subset Y$.
Then the canonical map
\[ \HHbig0.Y.\Omega_Y^p(\log E). \lto \HHbig0.\Reg X.\Omegap Xp. \]
is bijective for all $p \le \codim X{\Sing X} - 2$.
\end{thm}

Secondly, since we can prove it in characteristic zero slowly, we can find a counterexample in positive characteristic quickly.
Of course, nobody ever claimed that Flenner holds in characteristic $p$, but this means in particular that no counterexamples have been given so far.

\begin{thm}[Characteristic $p$ failure of Flenner's theorem] \label{flenner fail}
Fix an algebraically closed field $k$ of characteristic $p > 0$.
Then for any $n \in \set{3, 4}$, there exists an $n$-dimensional isolated singularity $\germ0X$ admitting a log resolution $\pi \from Y \to X$, with exceptional divisor $E$, and such that
\[ \HHbig0.Y.\Omega_Y^q(\log E). \lto \HHbig0.\Reg X.\Omegap Xq. \]
is \emph{not} surjective for any $1 \le q \le n - 2$.
\end{thm}

Similar examples also exist for higher values of $n$, assuming resolution of singularities in dimension $n - 1$.

\section*{Extension with logarithmic poles}

The only substantial ingredients to our proof of \cref{flenner log} (besides certain standard results) are \cref{neg lemma} below and the Bogomolov--Sommese vanishing theorem, asserting that on a projective snc pair $(X, D)$, any line bundle $\sL \subset \Omegal XpD$ has Kodaira dimension $\kappa(\sL) \le p$.
The latter statement follows from the closedness of logarithmic forms on $(X, D)$ and the branched covering trick.
While closedness is usually viewed as a consequence of Deligne's results on degeneracy of some spectral sequence~\cite{Deligne71}, a very short proof based on the classical theory of harmonic integrals has been given by Noguchi~\cite{Nog95}.
In the case $D = \emptyset$ (which, however, is not sufficient for our purposes), all one needs for closedness is actually Stokes' theorem.

\begin{prp}[Big Negativity Lemma] \label{neg lemma}
Let $\pi \from Y \to X$ be a projective bimeromorphic morphism of normal complex spaces.
Then for any nonzero effective $\pi$-exceptional \Q-Cartier divisor $E$ on $Y$, there is an irreducible component $P \subset E$ such that $-E\big|_P$ is $\pi\big|_P$-big (i.e.~the restriction to a general fibre of $\pi\big|_P$ is big).
\end{prp}

The proof is the same as for the algebraic version in~\cite[Prop.~4.1]{Gra12}, except that in the beginning, $H \subset Y$ should be chosen to be a relatively ample divisor (this exists because $\pi$ is assumed to be projective).

\subsection*{Proof of \cref{flenner log}}

Since any two log resolutions can be dominated by a third one, it suffices to prove the theorem for one particular choice of $\pi$.
By general results on resolution~\cite[Thm.~3.45]{Kol07}, we may therefore assume that $\pi$ is projective and an isomorphism over $\Reg X$.
Let $\sigma \in \HH0.\Reg X.\Omegap Xp.$ be a nonzero $p$-form, where $p \le c - 2$ with $c \defn \codim X{\Sing X}$.
By Grauert's Direct Image Theorem~\cite[Ch.~10, \S4]{CAS}, the sheaf $\pi_* \Omegal YpE$ is coherent.
The R{\"u}ckert Nullstellensatz~\cite[Ch.~3, \S2]{CAS} therefore applies to show that the pullback $\pi^* \sigma$ is a meromorphic form on $Y$.
Set $G$ to be the pole divisor of $\pi^* \sigma$ as a logarithmic form, i.e.~the minimal effective $\pi$-exceptional divisor such that
\begin{equation} \label{rueckert flenner}
\pi^* \sigma \in \HH0.Y.\Omegal YpE(G)..
\end{equation}
We want to show that $G = 0$.
Otherwise, there is a component $P \subset G$ such that $-G\big|_P$ is $\pi\big|_P$-big, by \cref{neg lemma}.
Set $P^c \defn (E - P)\big|_P$.
Note that~\labelcref{rueckert flenner} may be equivalently regarded as a map $\O Y (-G) \to \Omegal YpE$, which by minimality of $G$ does not vanish along $P$.
Denoting its restriction to $P$ by $i$, the residue sequence for $p$-forms along $P$ reads
\[ \xymatrix{
& & \O P \big( {-G\big|_P} \big) \ar@{ ir->}_-i[d] \ar@/^.5pc/[dr] \ar@/_.5pc/@{-->}[dl] \\
0 \ar[r] & \Omegal Pp{P^c} \ar[r] & \Omegal YpE\big|_P \ar^-{\res_P}[r] & \Omegal P{p-1}{P^c} \ar[r] & 0.
} \]
Hence $\O P \big( {-G\big|_P} \big)$ injects into $\Omegal Pr{P^c}$ either for $r = p - 1$ (if ${\res_P} \circ i \ne 0$) or for $r = p$ (if ${\res_P} \circ i = 0$).
In both cases, $r \le p$.

Now let $F \subset P$ be a general fibre of
\[ \rho \defn \pi\big|_P \from P \lto B \defn \pi(P), \]
and set $F^c \defn P^c\big|_F$.
Since $(P, P^c)$ is an snc pair, so is $(F, F^c)$. 
Furthermore, with $n = \dim X$, we have
\begin{equation} \label{dim F}
\dim F \ = \ \dim P - \dim B \ \ge \ n - 1 - (n - c) \ = \ c - 1
\end{equation}
since $B \subset \Sing X$.
By generic smoothness, 
we may shrink $B$ and assume that $\rho$ is an snc morphism of the pair $(P, P^c)$.
This means that $B$ is smooth and all components of $P^c$, as well as all their intersections (including $P$), are smooth over $B$.
In this case, there is a log differential sequence~\cite[Sec.~4.1]{EV90}
\[ 0 \lto \rho^* \Omegap B1 \lto \Omegal P1{P^c} \lto \Omegal {P/B}1{P^c} \lto 0 \]
inducing a filtration of $\Omegal Pr{P^c}$ with quotients~\cite[Ch.~II, Ex.~5.16]{Har77}
\[ \hspace{6.75em}
\sG_i \defn \rho^* \Omegap Bi \tensor \Omegal {P/B}{r-i}{P^c}, \qquad i = 0, \dots, r. \]
Consequently, $\O P \big( {-G\big|_P} \big)$ injects into $\sG_i$ for some $i$.
Restricting to $F$, we get $\sL \defn \O F \big( {-G\big|_F} \big) \inj \Omegal F{r-i}{F^c}^{\oplus \rk \Omegap Bi}$ and then even $\sL \inj \Omegal F{r-i}{F^c}$ by projecting onto a suitable summand.
In view of the Bogomolov--Sommese vanishing theorem~\cite[Cor.~6.9]{EV92}, this implies
\begin{equation} \label{BS}
\kappa(\sL) \ \le \ r - i \ \le \ r \ \le \ p.
\end{equation}
On the other hand, as $-G\big|_P$ is $\rho$-big, $\sL$ is a big line bundle on $F$ and so
\begin{equation} \label{big}
\kappa(\sL) \ = \ \dim F \ \ge \ c - 1
\end{equation}
thanks to~\labelcref{dim F}.
Put together, \labelcref{BS} and~\labelcref{big} yield $p \ge c - 1$, contradicting our assumption that $p \le c - 2$.
We conclude that $G = 0$, as desired. \qed

\section*{Counterexamples in positive characteristic}

A reformulation of Bogomolov--Sommese vanishing is that on an $n$-dimensional smooth projective variety $X$, we have $\HH0.X.\Omegap Xq \tensor L\inv. = 0$ for $q < n$ and any big line bundle $L$.
If $L$ is even assumed to be ample, the statement becomes a special case of (Kodaira--Akizuki--)Nakano vanishing~\cite[Thm.~1$''$]{AkizukiNakano54}.
Utilizing a construction due to Mumford~\cite{Mumford61}, we show that over fields of positive characteristic, Nakano vanishing fails in a strong sense.
Taking cones over these examples then allows us to prove \cref{flenner fail}.

\begin{prp}[Characteristic $p$ failure of Nakano vanishing] \label{Nakano fail}
Assume resolution of singularities in some fixed dimension $n \ge 2$ and characteristic $p > 0$.
Then over any algebraically closed field of characteristic~$p$, there exists an $n$-dimensional smooth projective variety $X$ together with an ample line bundle $L$ on $X$ such that
\[ \hspace{11em}
\HH0.X.\Omegap Xq \tensor L\inv. \ne 0 \qquad \text{for all $1 \le q \le n - 1$.} \]
In particular, there exist surfaces and threefolds with this property.
\end{prp}

\subsection*{Mumford's construction}

By a \emph{twisted differential $q$-form} on a variety $X$, we will mean a section of $\Omegap Xq \tensor M$ for some line bundle $M$ on $X$.
The following lemma is valid only in positive characteristic.

\begin{lem}[Regularizing forms] \label{reg forms}
Let $X$ be a smooth projective variety, $U \subset X$ a dense open subset and $\alpha \in \HH0.U.\Omegap Uq \tensor M\big|_U. \setminus \set0$ a nonzero rational twisted $q$-form.
Then there exists a finite separable map $f \from Y \to X$ such that $Y$ is normal and $f^* \alpha \in \HH0.Y.\Omegap Yq \tensor f^*M.$ is a nonzero \emph{regular} twisted $q$-form.
\end{lem}

\begin{proof}
Let $X = \bigcup U_i$ be a finite open covering of $X$ which trivializes $M$, i.e.~$M\big|_{U_i} \isom \O{U_i}$.
Using these trivializations, the $\alpha_i \defn \alpha\big|_{U_i}$ become ``ordinary'' rational $q$-forms.
As in~\cite[p.~340]{Mumford61}, we may find finite separable maps $f_i \from V_i \to U_i$ with $V_i$ smooth and $f_i^*(\alpha_i)$ regular.\footnote{Mumford only considers $1$-forms on surfaces, but an inspection of his argument reveals that it also applies to forms of higher degree.}
Consider the compositum $L$ of all the function fields $k(V_i)$, as a subfield of say the algebraic closure of $k(X) = k(U_i)$.
Then the extension $L/k(X)$ is finite and separable.
Take $Y$ to be the normalization of $X$ in~$L$ and $f \from Y \to X$ the natural map.
It is clear that $f^* \alpha$ is a regular twisted form.
Also, $f^* \alpha \ne 0$ because $\alpha \ne 0$ and $f$ is separable.
\end{proof}

\begin{proof}[Proof of \cref{Nakano fail}]
Pick an arbitrary smooth projective variety $Z$ of dimension~$n$, an ample line bundle $L_Z$ on $Z$ and nonzero rational $q$-forms $\alpha_q$ twisted by $M = L_Z\inv$, for each $1 \le q \le n - 1$.
(This obviously exists.)
By \cref{reg forms}, there is a finite cover $f \from W \to Z$ such that the $f^* \alpha_q$ are regular.
Note that $L_W \defn f^* L_Z$ is still ample.
Let $g \from X \to W$ be a resolution of $W$.
Then $L_X \defn g^* L_W$ is big and nef (but not ample unless $W$ is already smooth).
The forms $(f \circ g)^* (\alpha_q)$ are nonzero global sections of $\Omegap Xq \tensor L_X\inv$.

We may assume that $g$ is a composition of blow-ups of codimension two centers, in which case $L_X - E$ is ample for some suitable ($g$-anti-ample) effective $g$-exceptional \Q-divisor $E \ge 0$.
The trouble, of course, is that $E$ will almost never have integral coefficients and thus $L_X - E$ does not correspond to a line bundle.
We will overcome this difficulty by performing suitable (ramified) finite covers.

After a perturbation, we may assume that $E = \sum_P m_P P$ is a \zp-divisor, i.e.~if the coefficients $m_P = c_P/d_P$ are written in lowest terms, then $p$ does not divide $d_P$.
Let $P \subset \supp E$ be an arbitrary irreducible component, and set $d = d_P$.
By~\cite[Lemma~2.1]{BlochGieseker71}, after a finite cover we may assume that the line bundle $\O X(P)$ admits a $d$-th root.
The divisor $P$ may not be irreducible anymore, but it is still smooth.
If $s \in \HH0.X.\O X(P).$ is a section corresponding to $P$, consider
\[ \pi \from X[{\sqrt[d]s}] \lto X, \]
the cover obtained by extracting a $d$-th root of that section~\cite[(2.44)]{Kollar13}.
Since $X$ and $P$ are smooth, so is $X[{\sqrt[d]s}]$.
Furthermore, $\pi$ is totally ramified along $P$ by construction and hence $\pi^*(m_P P)$ becomes integral.
Also note that since $p$ does not divide $d$, all covers performed are in fact separable, so the twisted forms $\alpha_q$ stay nonzero.

Repeating this procedure finitely often, we arrive at a situation where $L \defn L_X - E$ becomes a \Z-divisor (and it stays ample).
The forms $\alpha_q \in \HH0.X.\Omegap Xq \tensor L_X\inv. \subset \HH0.X.\Omegap Xq \tensor L\inv.$ show that we have found the desired example, finishing the proof of \cref{Nakano fail}.
\end{proof}

\subsection*{Proof of \cref{flenner fail}}

Fix an integer $n \in \set{3, 4}$, and pick an $(n - 1)$-dimensional smooth projective variety $Y$ together with an ample line bundle $L$ violating Nakano vanishing as in \cref{Nakano fail}.
Let
\[ X \defn \Spec \bigoplus_{m \ge 0} \HH0.Y.L^m. \]
be the affine cone over $(Y, L)$.
Blowing up the vertex gives a log resolution $\pi \from \wt X \to X$, where $\wt X$ is the total space of $L\inv$ and the exceptional locus $E$ is the zero section.
As we have carefully proven in \cite[Prop.~11.3]{FunnyExt}, the non-vanishing of $\HH0.Y.\Omegap Yq \tensor L\inv.$ implies non-extension of $q$-forms on $X$.
Thus $X$ is an example proving \cref{flenner fail}. \qed

\providecommand{\bysame}{\leavevmode\hbox to3em{\hrulefill}\thinspace}
\providecommand{\MR}{\relax\ifhmode\unskip\space\fi MR}
\providecommand{\MRhref}[2]{%
  \href{http://www.ams.org/mathscinet-getitem?mr=#1}{#2}
}
\providecommand{\href}[2]{#2}

\end{document}